\documentclass[11pt]{amsart}
\usepackage{amsmath}
\usepackage{amssymb}
\usepackage{amsthm}
\usepackage{latexsym}
\usepackage{amscd}
\usepackage[all]{xy}

\usepackage{hyperref}
\usepackage{enumerate}

\newtheorem{thm}{Theorem}[section]
\newtheorem{cor}[thm]{Corollary}
\newtheorem{lem}[thm]{Lemma}
\newtheorem{prop}[thm]{Proposition}
\newtheorem{defn}[thm]{Definition}

\newtheorem{rem}[thm]{Remark}

\newtheorem{ex}[thm]{Example}

\providecommand{\norm}[1]{\left\| #1 \right\|}
\newcommand{\enuma}[1]{\begin{enumerate}[\textup{(}a\textup{)}] {#1} \end{enumerate}}

\newcommand{\bW}{{\mathbf W}}
\newcommand{\bI}{{\mathbf I}}

\newcommand{\mc}{\mathcal}
\newcommand{\mf}{\mathfrak}

\newcommand{\Z}{\mathbb Z}
\newcommand{\Q}{\mathbb Q}
\newcommand{\R}{\mathbb R}
\newcommand{\C}{\mathbb C}

\newcommand{\vareps}{\epsilon}

\newcommand{\matje}[4]{\left(\begin{smallmatrix} #1 & #2 \\ 
#3 & #4 \end{smallmatrix}\right)}

\def\ess{{\rm ess}}
\def\Irr{{\rm Irr}}
\def\Gal{{\rm Gal}}
\def\GL{{\rm GL}}
\def\Mat{{\rm M}}
\def\MP{{\rm MP}}

\def\PGL{{\rm PGL}}
\def\SL{{\rm SL}}
\def\JL{{\rm JL}}
\def\rL{{\rm L}}
\def\rN{{\rm N}}
\def\Nrd{{\rm Nrd}}
\def\St{{\rm St}}

\def\rec{{\rm rec}}
\def\der{{\rm der}}

\def\Ind{{\rm Ind}}
\def\cInd{{\rm c}\!-\!{\Ind}}

\def\Nor{{\rm N}}

\def\ab{{\rm ab}}

\begin{document}

\title{Depth and the local Langlands correspondence}

\author[A.-M. Aubert]{Anne-Marie Aubert}
\address{Institut de Math\'ematiques de Jussieu -- Paris Rive Gauche, 
U.M.R. 7586 du C.N.R.S., U.P.M.C., 4 place Jussieu 75005 Paris, France}
\email{anne-marie.aubert@imj-prg.fr}
\author[P. Baum]{Paul Baum}
\address{Mathematics Department, Pennsylvania State University,  University Park, PA 16802, USA}
\email{baum@math.psu.edu}
\author[R. Plymen]{Roger Plymen}
\address{School of Mathematics, Southampton University, Southampton SO17 1BJ,  England \emph{and} 
School of Mathematics, Manchester University, Manchester M13 9PL, England}
\email{r.j.plymen@soton.ac.uk \quad plymen@manchester.ac.uk}
\author[M. Solleveld]{Maarten Solleveld}
\address{Radboud Universiteit Nijmegen, Heyendaalseweg 135, 6525AJ Nijmegen, the Netherlands}
\email{m.solleveld@science.ru.nl}
\date{\today}
\thanks{The second author was partially supported by NSF grant DMS-1200475.}

\begin{abstract}   
Let $G$ be an inner form of a general linear group over a non-archimedean local field. 
We prove that the local Langlands correspondence for $G$ preserves depths. We also show
that the local Langlands correspondence for inner forms of special linear groups preserves
the depths of essentially tame Langlands parameters.
\end{abstract}

\maketitle
\tableofcontents

\section{Introduction}

Let $F$ be a non-archimedean local field and let $G$ be a connected reductive 
group over $F$. Let $\Phi(G)$ denote the collection of
equivalence classes of Langlands parameters for $G$, and $\Irr (G)$ the set 
of (isomorphism classes of) irreducible smooth $G$-representations.
A central role in the representation theory of such groups is played by the
local Langlands correspondence (LLC). It is supposed to be a map
\[
\Irr (G) \to \Phi (G)
\]
that enjoys several naturality properties \cite{Bor,Vog}. The LLC should preserve
interesting arithmetic information, like local L-functions and $\epsilon$-factors. 
A lesser-known invariant that makes sense on both sides of the LLC is \emph{depth}. 

The depth of a Langlands parameter $\phi$ is easy to define. 
For $r \in \R_{\geq 0}$ let $\Gal (F_s / F)^r$ be the $r$-th ramification subgroup 
of the absolute Galois group of $F$. Then the depth of $\phi$ is the smallest number 
$d(\phi) \geq 0$ such that $\phi$ is trivial on $\Gal (F_s/F)^r$ for all $r > d(\phi)$.

The depth $d(\pi)$ of an irreducible $G$-representation $\pi$ was defined by Moy and 
Prasad \cite{MoPr1,MoPr2}, in terms of filtrations $P_{x,r} (r \in \R_{\geq 0})$ of 
the parahoric subgroups $P_x \subset G$. On the basis of several examples (see below)
it is reasonable to expect that for many Langlands parameters $\phi \in \Phi (G)$
with L-packet $\Pi_\phi (G) \subset \Irr (G)$ one has
\begin{equation}\label{eq:D}
d(\phi) = d(\pi) \quad \text{ for all } \pi \in \Pi_\phi (G) .
\end{equation}
This relation would be useful for several reasons. Firstly, it allows one to employ 
some counting arguments in the local Langlands correspondence, because (up to 
unramified twists) there are only finitely many irreducible representations and 
Langlands parameters whose depth is at most a specified upper bound. 

Secondly, it would be a step towards a more explicit LLC. One can try to determine 
the groups $P_{x,r} / P_{x,r+\epsilon}$  $(\epsilon > 0$ small) and their 
representations explicitly, and to match them with representations of 
$\Gal (F_s / F) / \Gal (F_s / F)^{r + \epsilon}$.

Thirdly, one can use \eqref{eq:D} as a working hypothesis when trying to establish a 
local Langlands correspondence, to determine whether or not two irreducible 
representations stand a chance of belonging to the same L-packet.
\vspace{2mm}

The most basic case of depth preservation concerns Langlands parameters $\phi \in 
\Phi (G)$ that are trivial on both the inertia group $\bI_F$ and on $\SL_2 (\C)$. 
These can be regarded as Langlands parameters of negative depth. Such a $\phi$ is only 
relevant for $G$ if $G$ is quasi-split and splits over an unramified extension of $F$. 
In that case one can say that an irreducible $G$-representation has negative depth if 
it possesses a nonzero vector fixed by a hyperspecial maximal compact subgroup. 
The Satake isomorphism shows how to each such representation one can associate 
(in a natural way) a Langlands parameter of the above kind.

The $G$-representations of depth zero have been subjected to ample study, see for 
example \cite{GSZ,Mor,DBRe,Moe}. According to Moy--Prasad, an irreducible representation 
has depth zero if and only if it has nonzero vectors fixed by the pro-unipotent radical 
of some parahoric subgroup of $G$. This includes Iwahori-spherical representations and 
Lusztig's unipotent representations \cite{Lus1,Lus2}. A Langlands parameter has depth 
zero if and only if it is trivial on the wild inertia subgroup of the absolute Galois 
group of $F$. For depth zero the equality \eqref{eq:D} is conjectured, and proven in 
certain cases, in \cite{DBRe}. It fits very well with the aforementioned work of Lusztig.

In positive depth there is the result of Yu \cite[\S 7.10]{Yu2}, who proved \eqref{eq:D}
for unramified tori. For $\GL_n (F)$, \eqref{eq:D} was claimed in \cite[\S 2.3.6]{Yu1} 
and proved in \cite[Proposition 4.5]{ABPS3}. For $\mathrm{GSp}_4 (F)$, \eqref{eq:D} is 
proved in \cite[\S~10]{Gan2}. We refer to \cite{GrRe,Ree,ReYu} for some interesting 
examples of positive depth Langlands parameters and supercuspidal representations. 
Most of these examples satisfy \eqref{eq:D}, but in \cite[\S 7.4--7.5]{ReYu} some 
particular cases are mentioned in which \eqref{eq:D} does not hold. All these 
counterexamples appear in small residual characteristics.
  So it remains to be seen in which generality the local Langlands 
correspondence will preserve depths.
\vspace{2mm}

In this paper we will prove that the local Langlands correspondence preserves depth
for the inner forms of $\GL_n (F)$. In a few non-split cases, this was done before 
in \cite{LaRa}. For inner forms of $\SL_n (F)$, we will prove an inequality between
depths, which becomes an equality if the Langlands parameter is essentially tame
in the sense that it maps the wild inertia group to a maximal torus of $\PGL_n (\C)$.
Every Langlands parameter for an inner form of $\SL_n (F)$ is essentially tame if 
the residual characteristic of $F$ does not divide $n$.

Let $D$ be a division algebra with centre $F$, of dimension $d^2$ over $F$.
Then $G = \GL_m (D)$ is an inner form of $\GL_n(F)$ with $n = dm$. There is a 
reduced norm map Nrd$: \GL_m (D) \to F^\times$ and the derived group 
$G_{\der} := \ker (\text{Nrd} : G \to F^\times)$ is an inner form of $\SL_n (F)$.
Every inner form of $\GL_n (F)$ or $\SL_n (F)$ is isomorphic to one of this kind.

The main steps in the proof of our depth-preservation theorem are:
\begin{itemize}
\item With the Langlands classification one reduces the problem to essentially
square-integrable representations and elliptic Langlands parameters.
\item Express the depth in terms of $\epsilon$-factors and conductors. This is a
technical step which involves detailed knowledge of the representation theory
of $G$. Here it is convenient to use an alternative but equivalent
version of depth, the normalized level of an irreducible $G$-representation.
\item Show that the Jacquet--Langlands correspondence for $G = \GL_m (D)$
preserves $\epsilon$-factors. Since the LLC for $\GL_m (D)$ is defined as a 
composition of the Jacquet--Langlands correspondence with the LLC for $\GL_n (F)$  
and the latter is known to preserve
$\epsilon$-factors, this proves depth-preservation for inner forms of $\GL_n (F)$.
\item Relate the depth for $G_\der$ to depth for $G$. For irreducible 
representations nothing changes, but for Langlands parameters the depth can
decrease if one replaces the dual group $\GL_n (\C)$ by $\PGL_n (\C)$. 
Using several properties of the Artin reciprocity map, we show that such
a decrease in depth cannot occur if the Langlands parameter is essentially tame.
\end{itemize}

This paper develops results presented by the second author in a lecture at
the 2013 Arbeitstagung.

\smallskip
\textbf{Acknowledgements.}
The authors wish to thank Vincent S\'echerre for some helpful explanations
on the construction of simple types for $\GL_m(D)$, Mark Reeder for
pointing out some examples where the depth is not preserved, Wilhelm Zink
for explaining properties of the Artin reciprocity map, and Paul Broussous
for providing the reference \cite{BaBr}, which has allowed a substantial 
simplification of the proof of Proposition~\ref{prop:BF} from a previous version.

\section{The local Langlands correspondence for inner forms of $\GL_n(F)$} 
\label{sec:LLCGL}

\subsection{The statement of the correspondence} \

The local Langlands correspondence for supercuspidal representations of 
$\GL_n (F)$ was established in the important papers \cite{LRS,HaTa,Hen2}. 
Together with the Jacquet--Langlands correspondence this provides the LLC for 
essentially square-integrable representations of inner forms $G = \GL_m (D)$ of 
$\GL_n (F)$. It is extended to all irreducible $G$-representations via the Zelevinsky
classification \cite{Zel,DKV}, see \cite{HiSa,ABPS1}. For these groups every 
L-packet is a singleton and the LLC is a canonical bijective map
\begin{equation}\label{eq:1.1}
\mathrm{rec}_{D,m} : \Irr (\GL_m (D)) \to \Phi (\GL_m (D)) . 
\end{equation}
A remarkable aspect of Langlands' conjectures \cite{Vog} is that it is better
to consider not just one reductive group at a time, but all inner forms of
a given group simultaneously. Inner forms share the same
Langlands dual group, so in \eqref{eq:1.1} the right hand side is the same
for all inner forms $G$ of the given group. Then one can turn
\eqref{eq:1.1} into a bijection by defining a suitable equivalence
relation on
the set of inner forms and taking the corresponding union of the sets
$\Irr (G)$ on the left hand side (see Theorem~\ref{thm:1.1} below). 

We define the equivalence classes of such inner forms to be in bijection
with the isomorphism classes of central simple $F$-algebras of dimension $n^2$ 
via $\Mat_m (D) \mapsto \GL_m (D)$, respectively $\Mat_m (D) \mapsto \GL_m (D)_{\der}$.

As Langlands dual group of $G$ we take $\GL_n (\C)$. To deal
with inner forms it is advantageous to consider the conjugation action of
$\SL_n (\C)$ on these two groups. It induces a natural action of $\SL_n (\C)$
on the collection of Langlands parameters for $\GL_n (F)$.
For any such parameter $\phi$ we can define
\begin{equation}\label{eq:1.2}
C(\phi) = Z_{\SL_n (\C)}(\text{im } \phi) , \quad\text{and}\quad
\mc S_\phi = C(\phi) / C(\phi)^\circ.
\end{equation}
Notice that the centralizers are taken in $\SL_n (\C)$ and not in the
Langlands dual group. 

Via the Langlands correspondence the non-trivial irreducible representations of
$\mc S_\phi$ are associated to irreducible representations of
non-split inner forms of $\GL_n (F)$.
For example, consider a Langlands parameter $\phi$ for $\GL_2 (F)$ which is
elliptic, that is, whose image is not contained in any torus of $\GL_2 (\C)$.
Then $\mc S_\phi = Z(\SL_2 (\C)) \cong \{\pm 1\}$. The pair
$(\phi,\mathrm{triv}_{\mc S_\phi})$
parametrizes an essentially square-integrable representation of $\GL_2 (F)$ and
$(\phi,\mathrm{sgn}_{\mc S_\phi})$ parametrizes an irreducible
representation of the inner
form $D^\times$, where $D$ denotes a noncommutative division algebra of
dimension $4$ over $F$.

The enhanced version of the local Langlands correspondence for all inner forms
of general linear groups over nonarchimedean local fields says:

\begin{thm}\label{thm:1.1} \cite[Theorem~1.1]{ABPS3}  \\
There is a canonical bijection between:
\begin{itemize}
\item pairs $(G,\pi)$ with $\pi \in \Irr (G)$ and $G$ an inner form of $\GL_n (F)$,
considered up to equivalence;
\item $\GL_n (\C)$-conjugacy classes of pairs $(\phi,\rho)$ with 
$\phi \in \Phi (\GL_n (F))$ and $\rho \in \Irr (\mc S_\phi)$. 
\end{itemize}
\end{thm}

Via the Kottwitz isomorphism \cite[Proposition 6.4]{Kot} every character of $Z(\SL_n (\C))$
determines a central simple $F$-algebra $\Mat_m (D)$. As $Z(\SL_n (\C)) \subset C(\phi)$, 
for any Langlands parameter as above a character of $\mc S_\phi$ determines an
inner form $\GL_m (D)$ of $\GL_n (F)$. In contrast with the
usual LLC, our L-packets for inner forms of general linear groups need not be singletons. 
To be precise, the packet $\Pi_\phi$ contains the unique representation 
$\mathrm{rec}_{D,m}^{-1}(\phi)$ of $G = \GL_m (D)$ if $\phi$ is relevant for $G$, 
and no $G$-representations otherwise.\\[2mm]

\subsection{The Jacquet--Langlands correspondence} \

A representation $\pi$ of $G$ is called essentially square-integrable
if $\pi \big|_{G_{\der}}$ is square-integrable and there exists an unramified 
character $\chi$ of $G$ such that $\pi \otimes \chi$ is unitary. We denote the set of 
(equivalence classes of) irreducible essentially square-integrable $G$-representations 
by $\Irr_{\ess L^2}(G)$. There is a natural bijection between 
$\Irr_{\ess L^2}(\GL_n (F))$ and $\Irr_{\ess L^2}(\GL_m (D))$, discovered first for 
$\GL_2 (F)$ by Jacquet and Langlands \cite{JaLa}. The local Langlands correspondence 
for $\GL_m (D)$ is constructed with the help of the Jacquet--Langlands correspondence. 
Here we recall some useful properties of the latter correspondence.

\begin{thm}\label{thm:JL}
Let $\GL_m (D)$ be an inner form of $\GL_n (F)$. There exists a canonical bijection
\[
\JL : \Irr_{\ess L^2}(\GL_n (F)) \to \Irr_{\ess L^2}(\GL_m (D))
\]
with the following properties:
\enuma{
\item There is a canonical identification of the semisimple elliptic conjugacy
classes in $\GL_n (F)$ with those in $\GL_m (D)$. Let $g \in \GL_n (F)$ and $g' \in
\GL_m (D)$ be semisimple elliptic elements in the same conjugacy class and
let $\theta_\pi$ be the character of $\pi \in \Irr_{\ess L^2}(\GL_n (F))$. Then 
\[
(-1)^n \theta_\pi (g) = (-1)^m \theta_{\JL (\pi)} (g') . 
\]
\item $\JL$ preserves twists with characters of $F^\times$: 
\[
\JL (\pi \otimes \chi \circ \det) = \JL (\pi) \otimes \chi \circ \mathrm{Nrd} . 
\]
\item $\JL$ respects contragredients: $\JL (\pi^\vee) = \JL (\pi)^\vee$.
\item Let $P'$ be a standard parabolic subgroup of $\GL_m (D)$, with Levi factor 
$M' = \prod_i \GL_{m_i}(D)$. Let $P$ be the corresponding standard parabolic
subgroup of $\GL_n (F)$, with Levi factor $M = \prod_i \GL_{d m_i}(F)$. Then the 
Jacquet modules $r_P^{\GL_n (F)}(\pi)$ and $r_{P'}^{\GL_m (D)}(\JL (\pi))$ are 
either both zero or both irreducible and essentially square-integrable. 
In the latter case
\[
\JL \big( r_P^{\GL_n (F)}(\pi) \big) = r_{P'}^{\GL_m (D)}(\JL (\pi)) .
\]
In other words, $\JL$ and its version for $M$ and $M'$ respect Jacquet restriction.
\item $\JL$ preserves supercuspidality.
\item $\JL (\St_{\GL_n (F)}) = \St_{\GL_m (D)}$, where $\St_G$ denotes the 
Steinberg representation of $G$.
\item $\JL$ preserves $\gamma$-factors: 
\[
\gamma (s,\JL (\pi),\psi) = \gamma (s,\pi,\psi) \quad 
\text{for any nontrivial character } \psi \text{ of } F . 
\]
\item $\JL$ preserves $\rL$-functions: $L(s,\JL (\pi)) = L(s,\pi)$.
\item $\JL$ preserves $\epsilon$-factors: 
$\epsilon (s,\JL (\pi),\psi) = \epsilon (s,\pi,\psi)$.
}
\end{thm}
\begin{proof}
The correspondence, which is in fact characterized by property (a),
is proven over $p$-adic fields in \cite{DKV} and over local fields
of positive characteristic in \cite{Bad1}. 
Properties (b) and (c) are obvious in view of (a). The same goes for property 
(f) in the case $m=1$, because then $\St_{\GL_m (D)}$ is just the trivial 
representation of $D^\times$. For (d) see \cite[\S 5]{Bad1}, in particular
Proposition B. Obviously (d) implies (e). 
Property (f) for $m>1$ follows from (f) for $m=1$ and property (d).
Property (g) was proven over local function fields in \cite[p. 741]{Bad1},
with an argument that also works over $p$-adic fields. 

Properties (h) and (i) were claimed in \cite{DKV}, with the difference that 
the $\epsilon$-factors of $\pi$ and $\JL (\pi)$ are said to agree only up 
to a sign $(-1)^{n+m}$. This sign is due to a convention that does not agree 
with \cite{GoJa}, which we use for the definition of $\epsilon$-factors. 
Unfortunately the argument for (h) and (i) given in \cite[\S B.1]{DKV} 
is incorrect. Instead, we will establish (h) by direct calculation.

Let $\nu_D$ denote the unramified character $g' \mapsto \norm{\mathrm{Nrd}\, 
g'}_F$ of $\GL_m (D)$. Consider any $\pi' \in \Irr_{\ess L^2}(\GL_m (D))$. 
By \cite[\S B.2]{DKV} or \cite[\S 2]{Tad} there exist:
\begin{itemize}
\item integers $a,b,s_\sigma$ such that $ab=m$ and $s_\sigma$ divides $ad$;
\item an irreducible supercuspidal representation $\sigma$ of $\GL_a (D)$,
\end{itemize}
such that $\pi'$ is a consituent of the parabolically induced representation
\begin{equation}\label{eq:indrep}
\Pi' := I_{\GL_{a} (D)^b}^{\GL_m (D)} \Big( 
\nu_D^{s_\sigma \frac{1-b}{2}} \sigma \otimes 
\nu_D^{s_\sigma \frac{3-b}{2}} \sigma \otimes \cdots \otimes 
\nu_D^{s_\sigma \frac{b-1}{2}} \sigma \Big) .
\end{equation}
By \cite[Proposition 2.3]{Jac} the L-function of \eqref{eq:indrep} is the 
product of L-functions of the inducing representations:
\begin{equation}\label{eq:Lprod}
L(s,\Pi') = \prod_{k=1}^b L(s,\nu_D^{s_\sigma (k - (1+b)/2)} \sigma) =
\prod_{k=1}^b L(s + s_\sigma (k - (1+b)/2), \sigma) .
\end{equation}
By definition $L(s,\pi')^{-1}$ is a monic polynomial in $q^{-s}$, and by
\cite[2.7.4]{Jac} it is a factor of the monic polynomial $L(s,\Pi')^{-1}$.
Now there are two cases to be distinguished, depending on whether $\sigma$
is an unramified representation of $D^\times$ or not.

\textbf{Case 1:} $a = 1, b = m$ and $\sigma$ is unramified. \\
There exists an
unramified character $\chi$ of $F^\times$ such that $\sigma = \chi \circ $Nrd.
By \cite[\S B.2]{DKV} or \cite[\S 2]{Tad} \eqref{eq:indrep} only has an 
essentially square-integrable subquotient if $s_\sigma = d$. Then 
$\pi' \cong \St_{\GL_m (D)} \otimes \chi \circ $Nrd, so $\JL^{-1}(\pi') \cong
St_{\GL_n (F)} \otimes \chi \circ \det$. With property (f) this enables us
to compute the $\gamma$-factor. Let $\omega_F$ be a uniformizer of $F$,
$\mf o_F$ the ring of integers and $\mf p_F$ its maximal ideal.
Assume that $\psi$ is trivial on $\mf p_F$ but not on $\mf o_F$. Then
\begin{equation}\label{eq:gammaSt}
\gamma(s,\pi',\psi) = \gamma(s,\St_{\GL_n (F)} \otimes \chi \circ \det,\psi) =
(-1)^n q^{n/2} \frac{1 - q^{-s + (1-n)/2} \chi (\omega_F)}{1 - q^{-s + (1+n)/2} 
\chi (\omega_F)} .
\end{equation}
By \cite[Proposition 4.4]{GoJa}, \eqref{eq:Lprod} becomes
\begin{equation}\label{eq:Lprod2}
\prod_{k=1}^m L(s + d (k - (1+m)/2), \chi \circ \mathrm{Nrd}) = 
\prod_{k=1}^m L(s + d (k - (1+m)/2) + (d-1)/2, \chi ) .
\end{equation}
Now we apply \cite[Theorem 7.11.4]{GoJa}. It is stated only for $\GL_n (F)$,
but the proof with zeros and poles of L-functions goes through because we 
know $\gamma (s,\pi',\psi)$. We find that for the L-function of 
$\pi'$ we need only the factor $k=m$ of \eqref{eq:Lprod2}:
\[
L(s,\pi',\psi) = L(s + (n-1)/2,\chi) = 
(1 - q^{-s + (1-n)/2} \chi (\omega_F) )^{-1} . 
\]
In particular the whole calculation works with $d=1$, so
\begin{equation}\label{eq:LSt}
L(s,\St_{\GL_m (D)} \otimes \chi \circ \mathrm{Nrd}) = 
L(s,\St_{\GL_n (F)} \otimes \chi \circ \det) = L(s + (n-1)/2,\chi) .
\end{equation}

\textbf{Case 2:} all other $\sigma$. \\
Then \cite[Proposition 5.11]{GoJa}
says that $L(s,\sigma \otimes \chi) = 1$ for every unramified character
$\chi$ of $\GL_a (D)$. Hence $L(s,\Pi') = 1$ by \eqref{eq:Lprod}. We observed
above that $L(s,\pi')^{-1}$ is a factor of $L(s,\Pi')$, so $L(s,\pi') = 1$.
Because JL is bijective, $\JL^{-1}(\pi')$ is not an unramified twist of the
Steinberg representation, so $L(s,\JL^{-1}(\pi')) = 1$ as well. This
proves property (h).

In view of the relation
\begin{equation}\label{eq:epsilonLgamma}
\epsilon (s,\pi,\psi) = \gamma (s,\pi,\psi) L(s,\pi) L(1-s,\pi^\vee), 
\end{equation}
(i) follows directy from (c), (g) and (h).
\end{proof}

We record a particular consequence of equations \eqref{eq:gammaSt},
\eqref{eq:LSt} and \eqref{eq:epsilonLgamma}:
\begin{equation}\label{eq:epsilonSt}
\epsilon (s,\St_{\GL_m (D)} \otimes \chi \circ \mathrm{Nrd},\psi) = 
(-1)^{n-1} \epsilon (s,\chi,\psi) = (-1)^{n-1} q^{s-1/2} \chi (\omega_F^{-1}) 
\end{equation}
for any character $\psi$ of $F$ which is trivial on $\mf p_F$ but
not on $\mf o_F$.

\subsection{Depth for Langlands parameters} 
\label{sec:depth} \

Let $F_s$ be a separable closure of $F$ and let $\Gal (F_s / F)$ be the
absolute Galois group of $F$. We recall some properties of its 
ramification groups (with respect to the upper numbering), as defined
in \cite[Remark IV.3.1]{Ser}:
\begin{itemize}
\item $\Gal (F_s/F)^{-1} = \Gal (F_s / F)$ and $\Gal (F_s / F)^0 =
\bI_F$, the inertia group.
\item  For every $l \in \R_{\geq 0}, \; \Gal (F_s / F)^l$ is a compact 
subgroup of $\mathbf \bI_F$. It consists of all $\gamma \in \Gal (F_s / F)$ 
which, for every finite Galois extension $E$ of $F$ contained in $F_s$,
act trivially on the ring $\mf o_E / \mf p_E^{i(l,E)}$ 
(where $i(l,E) \in \Z_{\geq 0}$ can be found with \cite[\S IV.3]{Ser}).
\item $l \in \R_{\geq 0}$ is called a jump of the filtration if
\[
\Gal (F_s / F)^{l+} := \bigcap_{l' > l} \Gal (F_s / F)^{l'}
\]
does not equal $\Gal (F_s / F)^l$. The set of jumps of the filtration is
countably infinite and need not consist of integers.
\end{itemize}

Recall \cite{Bor} that a Langlands parameter for $\GL_m (D)$ is a continuous 
homomorphism 
\[
\phi : \mathbf W_F \times \SL_2 (\C) \to \GL_n (\C)
\]
such that:
\begin{itemize}
\item $\phi (\mathbf W_F)$ consists of semisimple elements;
\item $\phi \big|_{\SL_2 (\C)} : \SL_2 (\C) \to \GL_n (\C)$ is a morphism
of complex algebraic groups;
\item $\phi$ is relevant for $\GL_m (D)$. This means that the conjugacy class of
a Levi subgroup of $\GL_n (\C)$ minimally containing im$(\phi)$ should 
correspond to a conjugacy of class of Levi subgroups of $\GL_m (D)$.
\end{itemize}
We define the depth of such a Langlands parameter as 
\[
d(\phi) := \inf \{ l \geq 0 \mid \Gal (F_s/F)^{l+} \subset \ker \phi \} . 
\]
We say that $\phi\in\Phi(\GL_n(F))$ is \emph{elliptic} if
its image is not contained in any proper Levi subgroup of $\GL_n (\C)$.

Let $\psi$ be a nontrivial character of $F$ and let $c(\psi)$ be the 
largest integer $c$ such that $\psi$ is trivial on $\mf p_F^{-c}$.
The $\epsilon$ factor of $\phi$ (and $\psi$) was defined in \cite{Tat}.
It takes the form
\begin{equation}\label{eq:epsilonPhi}
\epsilon (s,\phi,\psi) = \epsilon (0,\phi,\psi) q^{-s (a(\phi) + 
n c (\psi))} \text{ with } \epsilon (0,\phi,\psi) \in \C^\times .
\end{equation}
Here $a(\phi) \in \Z_{\geq 0}$ is the Artin conductor of $\phi$ 
(called $f(\phi)$ in \cite[\S VI.2]{Ser}). To study $a(\phi)$ it is
convenient to rewrite $\phi$ in terms of the Weil--Deligne group. 
For $\gamma \in \mathbf W_F$ put
\begin{equation}\label{eq:phi0}
\phi_0 (\gamma) = \phi (\gamma,1) \phi \Big( 1,
\matje{\norm{\gamma}^{1/2}}{0}{0}{\norm{\gamma}^{-1/2}} \Big) ,
\end{equation}
so $\phi_0$ is a representation of $\mathbf W_F$ which agrees with
$\phi$ on $\bI_F$. Define $N \in \mf{gl}_n (\C)$ as the nilpotent 
element $\log \phi \big(1, \matje{1}{1}{0}{1} \big)$. Then $(\phi_0,N)$
is the Weil--Deligne representation of $\mathbf W_F \ltimes \C$ 
corresponding to $\phi$. 

Denote the vector space $\C^n$ endowed with the representation $\phi$
by $V$, and write $V_N = \ker (N : V \to V)$. 
By definition \cite[\S 4.1.6]{Tat}
\begin{align}\label{eq:aphi}
& a (\phi) = a(\phi_0) + \dim (V^{\bI_F} / V_N^{\bI_F}) ,\\
& \epsilon (s,\phi,\psi) = \epsilon (0,\phi_0,\psi) \det \Big( -\text{Frob} 
\big|_{V^{\bI_F} / V_N^{\bI_F}} \Big) q^{-s (a(\phi) + n c(\psi))} ,
\end{align}
where Frob denotes a geometric Frobenius element of $\mathbf W_F$.

\begin{lem} \label{lem:phi} 
For any elliptic $\phi \in \Phi (GL_n (F))$
\begin{equation} \label{eqn:depth}
d(\phi):=\begin{cases}
0&\text{if $\bI_F\subset\ker(\phi)$,}\cr
\frac{\displaystyle a (\phi)}{\displaystyle n}-1 &\text{otherwise,}
\end{cases}\end{equation}
\end{lem}
\begin{proof}
This was proved in \cite[Lemma~4.4]{ABPS3} under the additional
assumption $\SL_2 (\C) \subset \ker \phi$. We will reduce to that
special case.

Since $\phi$ is elliptic, it defines an irreducible $n$-dimensional
representation $V$ of $\mathbf W_F \times \SL_2 (\C)$. Hence there are
irreducible representations $(\phi_1,V_1)$ of $\mathbf W_F$ and
$(\phi_2,V_2)$ of $\SL_2 (\C)$ such that
\begin{equation}\label{eq:tensorV}
(\phi,V) = (\phi_1,V_1) \otimes (\phi_2,V_2).
\end{equation}
In particular $V^{\bI_F} = V_1^{\bI_F} \otimes V_2$. Suppose first that
$V_1^{\bI_F} = V_1$. Then $\bI_F \subset \ker \phi$, so $d(\phi) = 0$ 
by definition. Now suppose $V_1^{\bI_F} \neq V_1$. As $(\phi_1,V_1)$ is
irreducible and $\bI_F$ is normal in $\mathbf W_F$, we must have
$V_1^{\bI_F} = 0$. Hence $V^{\bI_F} = 0$, which by \eqref{eq:phi0} and 
\eqref{eq:aphi} implies $a(\phi) = a(\phi_0)$. By \cite[Corollary VI.2.1']{Ser} 
$a(\phi_0)$ is additive in $V$ and depends only on
\[
\phi_0 \big|_{\bI_F} = \phi \big|_{\bI_F} = 
\phi_1 \big|_{\bI_F} \otimes \mathrm{id}_{V_2} .
\]
Now it follows from \eqref{eq:tensorV} that
\begin{equation}\label{eq:aphi1}
a (\phi) = a(\phi_1) \dim V_2 = n a (\phi_1) / \dim V_1 . 
\end{equation}
As $\ker \phi_1 \supset \SL_2 (\C)$ we may apply \cite[Lemma~4.4]{ABPS3},
which together with \eqref{eq:aphi1} gives
\[
d(\phi_1) = \frac{a(\phi_1)}{\dim V_1} - 1 = \frac{a(\phi)}{n} - 1 . 
\]
To conclude, we note that $d(\phi_1) = d(\phi)$ by \eqref{eq:tensorV}.
\end{proof}

\subsection{The depth of representations of $\GL_m(D)$}
\label{subsec:DepthCond} \

Let $k_D=\mf o_D/\mf p_D$ be the residual field of $D$. 
Let $\mf A$ be a hereditary $\mf o_F$-order in $\Mat_m(D)$. The
Jacobson radical of $\mf A$ will be denoted by $\mf P$. Let $r=e_D(\mf A)$ and 
$e=e_F(\mf A)$ denote the integers defined by $\mf p_D\mf A=\mf P^r$ and 
$\mf p_F\mf A=\mf P^e$, respectively. We have  
\begin{equation} \label{eqn:eFeD}
e_F(\mf A)=d\,e_D(\mf A).
\end{equation}
The normalizer in $G$ of $\mf A^\times$ will be denoted by
\[
\mf K(\mf A):=\left\{g\in G \,:\,g^{-1}\mf A^\times g= \mf
A^\times\right\}.
\]
Define a sequence of compact open subgroups of $G=\GL_m(D)$ by
\[
U^0(\mf A):=\mf A^\times,\quad\text{and}\quad
U^j(\mf A):=1+\mf P^j,\;\;j\ge 1.
\]
Then $\mf A^\times$ is a parahoric subgroup of $G$ and $U^1 (\mf A)$ is
its pro-unipotent radical. We define the \emph{normalized level} of an 
irreducible representation $\pi$ of $G$ to be
\begin{equation} \label{eqn:def_depth}
d(\pi):=\min\left\{j/e_F(\mf A)\right\},
\end{equation}
where $(j,\mf A)$ ranges over all pairs consisting of an integer $j\ge 0$
and a hereditary $\mf o_F$-order $\mf A$ in $\Mat_m(D)$ such that $\pi$ contains
the trivial character of $U^{j+1}(\mf A)$. 

\begin{rem}{\rm
When $\pi$ is a representation of $\GL_n (F)$, our notion of normalized level 
coincides with that of \cite[\S~12.6]{BuHe}. However when $\pi$ is a 
representation of $D^\times$ (or more generally of $\GL_m(D)$), the normalized level
of $\pi$ as defined above 
is not equal to the level $\ell_D(\pi)$ defined 
in \cite[\S~54.1]{BuHe} (resp. $\ell(\pi)$ defined by Broussous in \cite[Th\'eor\`eme~A.1.2]{BaBr}): we have 
\[d(\pi)=\frac{1}{d}\,\ell_D(\pi) \quad(\text{resp. }\;\; 
d(\pi)=\frac{1}{d}\,\ell(\pi).\] 
This reflects the fact that we have divided by $e_F(\mf A)$ instead 
of $e_D(\mf A)$ in Eqn.~\eqref{eqn:def_depth}. }
\end{rem}
 
The following proposition will allow to use both results that were written in the
setting of the normalized level, as general results on the depth in the sense
of Moy and Prasad.

\begin{prop}\label{prop:compareDepth}
The normalized level of $\pi \in \Irr (G)$ equals its Moy--Prasad depth.
\end{prop}
\begin{proof}
Let us denote the Moy--Prasad depth of $(\pi,V_\pi)$ by $d_{\MP}(\pi)$ for the duration
of this proof. For any point $x$ of the Bruhat--Tits building $\mc B (G)$ of $G$, 
consider the Moy--Prasad filtrations $P_{x,r}, P_{x,r+} \; (r \in \R_{\geq 0})$ of the 
parahoric subgroup $P_x \subset G$ \cite[\S 2]{MoPr1}. We normalize
these filtrations by using the valuation on $F_s$ which maps $F^\times$ onto $\Z$.
Then $d_{\MP}(\pi)$ is the minimal $r \in \R_{\geq 0}$ such that $V_\pi^{P_{x,r+}} \neq 0$
for some $x \in \mc B (G)$, see \cite[\S 3.4]{MoPr2}.

Any hereditary $\mf o_F$-order $\mf A$ in $\Mat_m (D)$ is associated to a unique facet
$\mc F( \mf A)$ of $\mc B (G)$. The filtration $\{ U^j (\mf A) \mid j \in \Z_{\geq 0} \}$
was compared with the Moy--Prasad groups for $x \in \mc F (\mf A)$ by Broussous and
Lemaire. Let $x_{\mf A}$ be the barycenter of $\mc F (\mf A)$. From  
\cite[Proposition 4.2 and Appendix A]{BrLe} and the definition of $e_F (\mf A)$ we see that
\[
U^j (\mf A) = P_{x_{\mf A},j / e_F (\mf A)} \text{ for all } j \in \Z_{\geq 0} .
\]
Hence the definitions of the normalized level and the Moy--Prasad depth are almost 
equivalent, the only difference being that for $d_{\MP}(\pi)$ we must consider all points
of $\mc B (G)$, whereas for $d(\pi)$ we may only use barycenters of facets of $\mc B (G)$.
Thus it remains to check the following claim:
there exists a facet $\mc F$ of $\mc B (G)$ with barycenter $x_{\mc F}$, such that $V_\pi$
has nonzero $P_{x_{\mc F},d_{\MP}(\pi)+}$-invariant vectors.

This is easy to see with the explicit constructions of the groups $P_{x,r}$ at hand, but
we prefer not to delve into those details here. In fact, since every chamber of $\mc B (G)$
intersects every $G$-orbit in $\mc B (G)$, it suffices to consider facets contained in the 
closure of a fixed "standard" chamber. Then the claim becomes equivalent to saying that 
$x_{\mc F}$ is an "optimal point" in the sense of \cite[\S 6.1]{MoPr1}. That is assured by
\cite[Remark 6.1]{MoPr1}, which is applicable because the root system of $G$ is 
of type $A_{m-1}$. 
\end{proof}

\subsection{Conductors of representations of $\GL_m (D)$} \

Let $\vareps(s,\pi,\psi)$ denote the Godement--Jacquet local constant \cite{GoJa}.
It takes the form
\begin{equation} \label{eqn:epsGJ}
\vareps(s,\pi,\psi) =
\vareps(0,\pi,\psi) \, q^{-f(\pi,\psi)s}, \quad\text{where } \vareps(0,\pi,\psi)
\in \C^\times. 
\end{equation}
Recall that $c(\psi)$ is the largest integer $c$ such that 
$\mf p_F^{-c}\subset\ker\psi$. In the previous section we had $c(\psi)=-1$.

A representation of $D^\times$ is called \emph{unramified} if it is trivial on 
$\mf o_D^\times$. An unramified representation of $D^\times$ is a character 
and has depth zero.

\begin{prop} \label{prop:BF}
Let $\pi$ be a supercuspidal irreducible representation of $G$. 
We have
\begin{equation} \label{eqn:BF}
f(\pi,\psi)\,=\,\begin{cases}
n\left(c(\psi)+1\right)-1&\text{if $m=1$ and $\pi$ is unramified,}\cr
n\,\left(d(\pi)+1+c(\psi)\right)&\text{otherwise.}\end{cases}
\end{equation}
\end{prop}
\begin{proof}
We suppose first that $m=1$ (so $d=n$) and $\pi$ is unramified. The required 
formula can be read off from \eqref{eq:epsilonSt} if $c(\psi) = -1$. For 
general $\psi$, applying \cite[Theorem~3.2.11]{BuFrLN} and taking in account 
\cite[(1.2.7), (1.2.8), (1.2.10)]{BuFrLN}, we obtain
\[
f(\pi,\psi)=\left(d(1-d-dc(\psi)\right)\,\cdot\, (-\frac{1}{d})=d+dc(\psi)-1.
\] 
Hence the first case of Eqn.~\eqref{eqn:BF} holds.

From now on, we will assume that $m\ge 2$ or $\pi$ is ramified. Then by combining
\cite[Th\'eor\`eme~A.2.1]{BaBr} with the 
fact that the Godement--Jacquet L-function $L(s,\pi)$ is $1$, we see that $\pi$ satisfies the conditions of Theorem 3.3.8 of
\cite{BuFr}.  Recall that $n=md$. 
By applying the formula of \cite[Theorem~3.3.8~(iv)]{BuFr},
we obtain
\[q^{f(\pi,\psi)}=\left[\mf A:\mf p_F^{c(\psi)+1}\mf P^{j}\right]^{1/n}.\]
On the other hand, the $\mf o_F$-order is $G$-conjugate to
the standard principal $\mf o_F$-order of $\Mat_{m}(D)$ defined by
the partition $(t,\ldots,t)$ ($r$-times) of $m$, where $m=rt$ and
$r=e_D(\mf A)$. Hence we have $\mf A/\mf P\simeq
\left(\Mat_t(k_D)\right)^r$. It follows that
\[
\left[ \mf A : \mf P \right] = (q^d)^{t^2r} = q^{drt^2}.
\]
Hence we get 
\[
f(\pi,\psi)=\frac{drt^2(j+e+ec(\psi))}{n}={n\left(\frac{j}{e}+1+c(\psi)\right)},
\]
since $drt^2=nt=n^2/e$. 

On the other hand it follows from \cite[Corollaire~5.22]{SeSt4} that there exists a maximal simple
type $(J,\lambda)$ in $G$, and an extension $\Lambda$ of $\lambda$ to the
normalizer $\bar J=\Nor_G(\lambda)$ of $\lambda$, such that 
\[
\pi=\cInd_{\bar J}^G\Lambda.
\]
By the construction of the type $(J,\lambda)$, we have
$d(\pi)\le j/e$. Conversely, let $[\mf A',j',j'-1,\beta']$ be a stratum
contained in $\pi$.
Then if $[\mf A',j',j'-1,\beta']$ is such that its
normalized level $j'/e'$ is minimal among the normalized levels of all the strata
contained in $\pi$, it is necessarily fundamental
\cite[Theorem~1.2.1.~(ii)]{Bro}. 
Since all the fundamental strata contained in $\pi$ have the same normalized level \cite[Th\'eor\`eme~A.1.2]{BaBr}, we get
 $j/e=d(\pi)$.  
\end{proof}

Theorem \ref{thm:depth_cond} below proves the validity of Conjecture 4.3
of \cite{LaRa}. 
In the case when $F$ has characteristic $0$, it is due to Lansky and Raghuram
for the groups $\GL_n(F)$ and $D^\times$, \cite[Theorem~3.1]{LaRa}, and for
certain representations of $\GL_2(D)$, \cite[Theorem~4.1]{LaRa}.  
Our proof is inspired by those of these results.
 
\begin{thm} \label{thm:depth_cond}
The depth $d(\pi)$ and the conductor $f(\pi):=f(\pi,\psi)-nc(\psi)$ of
each essentially square-integrable irreducible representation $\pi$ of 
$\GL_m(D)$ are linked by the following relation: 
\begin{equation} \label{eqn:dc}
d(\pi)=\begin{cases}
0&\text{if $\pi$ is an unramified twist of $\St_{\GL_m (D)}$}\cr
\frac{\displaystyle f(\pi)-n}{\displaystyle n}&\text{otherwise.}
\end{cases},\end{equation}
In particular
\begin{equation} \label{eqn:dcII}
d(\pi)=\max\left\{\frac{f(\pi)-n}{n},0\right\}.\end{equation} 
\end{thm}
\begin{proof}
Let $\pi \in \Irr_{\ess L^2}(\GL_m (D))$. We use the same notation as for $\pi'$ 
in the proof of Theorem \ref{thm:JL}.h, so $\pi$ is consituent of 
\[
I_{\GL_{a} (D)^b}^{\GL_m (D)} \Big( 
\nu_D^{s_\sigma{\frac{(1-b)}{2}}} \sigma \otimes 
\nu_D^{s_\sigma{\frac{(3-b)}{2}}} \sigma \otimes \cdots \otimes 
\nu_D^{s_\sigma{\frac{(b-1)}{2}}} \sigma \Big) ,
\] 
where $\sigma \in \Irr (\GL_a (D))$ is supercuspidal.
Since the depth is preserved by parabolic induction
\cite[Theorem 5.2]{MoPr2}, we get
\[
d(\pi) = d \Big( \nu_D^{s_\sigma{\frac{(1-b)}{2}}} \sigma \otimes 
\nu_D^{s_\sigma{\frac{(3-b)}{2}}} \sigma \otimes \cdots \otimes 
\nu_D^{s_\sigma{\frac{(b-1)}{2}}} \sigma \Big) .
\]
It follows that
\begin{equation} \label{dpisigma}
d(\pi) = d(\sigma).
\end{equation}
We will apply Proposition~\ref{prop:BF} to the supercuspidal representation 
$\sigma$ of $\GL_a(D)$. In the special case $\sigma$ is an unramified
representation of $D^\times$ (hence $a=1$ in this case), 
Eqn. \eqref{eqn:BF} gives 
\[f(\sigma,\psi)=d\left(c(\psi)+1\right)-1,\] 
that is, $f(\sigma)=d-1$. Hence we get
\[ \frac{f(\sigma)-d}{d} = -\frac{1}{d}.
\]
Then it implies that
\[\max\left\{\frac{f(\sigma)-d}{d},0\right\}=
\max\left\{-\frac{1}{d},0\right\}=0=d(\sigma),\]
in other words, Eqn.~\eqref{eqn:dcII} 
holds for the unramified representations of $D^\times$.

In the other cases (that is, $a\ne 1$ or $\sigma$ is ramified),
\eqref{eqn:BF} gives $f(\sigma)=ad(d(\sigma)+1)$, that is,
\begin{equation} \label{eqn:df}
\frac{f(\sigma)}{ad}=d(\sigma)+1.
\end{equation}
Since $d(\sigma)\ge 0$ (by definition of the depth), we obtain that
\begin{equation} \label{dsigma}
d(\sigma)=\max\left\{\frac{f(\sigma)-ad}{ad},0\right\}.
\end{equation}
Hence \eqref{eqn:dc} holds for every supercuspidal irreducible representation 
of $\GL_a(D)$, with $a\ge 1$ an arbitrary integer.

Recall that $s_\sigma$ is an integer
dividing $ad$, say $ad=a^*s_\sigma$ with $a^* \in \Z$. The 
image $\JL^{-1}(\sigma)$ of $\sigma$ under the  
Jacquet-Langlands correspondence is equivalent to the
Langlands quotient of the parabolically induced representation
\[I_{\GL_{a^*} (F)^{s_\sigma}}^{\GL_{a^*s_\sigma} (F)} \big( 
\nu_F^{\frac{(1-s_\sigma)}{2}} \sigma^* \otimes 
\nu_F^{\frac{(3-s_\sigma)}{2}} \sigma^* \otimes \cdots \otimes 
\nu_F^{\frac{(s_\sigma-1)}{2}} \sigma^* \big) ,
\]
where $\sigma^*$ is a unitary supercuspidal irreducible representation of
$\GL_{a^*}(F)$ and $\nu_F (g^*) = |\det (g^*)|_F$.

The representation $\JL^{-1} (\pi)$ is equivalent to a constituent of the 
parabolically induced representation
\[I_{\GL_{a^*} (F)^{b s_\sigma}}^{\GL_{adb} (F)} \big( 
\nu_F^{\frac{(1-b s_\sigma)}{2}} \sigma^* \otimes 
\nu_F^{\frac{(3-b s_\sigma)}{2}} \sigma^* \otimes \cdots \otimes 
\nu_F^{\frac{(b s_\sigma -1)}{2}} \sigma^* \big).
\]
We recall from \cite[\S~2.6]{Hen0} the formula describing the epsilon factor 
of $\JL^{-1} (\pi)$ in terms of the local factors of $\sigma^*$:
\begin{equation} \label{eqn:epsJLpi}
\epsilon(s,\JL^{-1} (\pi),\psi)=\prod_{i=0}^{b s_\sigma-1}
\epsilon(s+i,\sigma^*,\psi)\,\prod_{j=0}^{b s_\sigma-2}
\frac{L(-s-j,\check{\sigma^*})}{L(s+j,\sigma^*)}.
\end{equation}
Since the Jacquet--Langlands correspondence preserves 
the $\epsilon$-factors (see Theorem \ref{thm:JL}~i) we have 
\[
\epsilon(s,\JL^{-1}(\pi),\psi)=\epsilon(s,\pi,\psi).
\] 
Thus we have obtained the following formula
\begin{equation} \label{eqn:epspi}
\epsilon(s,\pi,\psi)=\prod_{i=0}^{b s_\sigma-1}
\epsilon(s+i,\sigma^*,\psi)\,\prod_{j=0}^{b s_\sigma-2}
\frac{L(-s-j,\check{\sigma^*})}{L(s+j,\sigma^*)}.
\end{equation}
If $\pi=\St_{\GL_m (D)} \otimes \chi$ for some unramified character 
$\chi$ of $D^\times$, it follows from \eqref{eq:epsilonSt} that  
$f(\pi,\psi)=-1$ in the case where $c(\psi)=-1$,
hence we obtain
\begin{equation} \label{eqn:fSt}
f(\pi) = n-1.
\end{equation}
From now on we assume $\pi$ is not equivalent to a representation of the form 
$\St_{\GL_m (D)} \otimes \chi$, with $\chi$ an unramified character of 
$D^\times$ (that is, we have $m\ne 1$ or $\sigma$ ramified). Then 
Theorem \ref{thm:JL}~b and~f implies that similarly $\JL^{-1}(\pi)$ is not a 
twist of $\St_{\GL_n (F)}$ by an unramified character of $F^\times$. Thus we 
have $a^*\ne 1$ or $\sigma^*$ ramified. It follows that 
$L(-s-j,\check{\sigma^*})=L(s+j,\sigma^*)=1$, and we obtain from (\ref{eqn:epspi}) that
\begin{equation} \label{eqn:fpi}
f(\pi)=b s_\sigma\,f(\sigma^*).
\end{equation}
In the special case when $b=1$ the equation (\ref{eqn:fpi}) gives
\begin{equation} \label{eqn:fsigma}
f(\sigma)= s_\sigma\,f(\sigma^*).
\end{equation}
Then using (\ref{dpisigma}) and (\ref{dsigma}) we get
\begin{equation} \label{eqn:dpi}
d(\pi)=d(\sigma)=\max\left\{\frac{b s_\sigma f(\sigma^*)-bad}{bad},0\right\}
=\max\left\{\frac{f(\pi)-n}{n},0\right\}. \qedhere
\end{equation}
\end{proof}

\subsection{Depth preservation} \

\begin{cor}
The Jacquet--Langlands correspondence preserves the depth of 
essentially square-integrable representations of $\GL_m (D)$. 
\end{cor}
\begin{proof}
Theorem \ref{thm:JL}.i shows in particular that the Jacquet--Langlands 
correspondence preserves conductors. Now Theorem \ref{thm:depth_cond} shows 
that it preserves depths as well.
\end{proof}

Theorems \ref{thm:depth_cond} and \ref{thm:JL} are also the crucial steps to
show that the local Langlands correspondence for inner forms of $\GL_m (D)$ preserves
depths. With similar considerations we show that it also preserves L-functions, 
$\epsilon$-factors and $\gamma$-factors. We abbreviate these three to "local factors".
For the basic properties of the local factors of Langlands parameters we refer to 
\cite{Tat}.

\begin{thm} \label{thm:LLCdepthcusp}
The local Langlands correspondence for representations of $\GL_m(D)$
preserves L-functions, $\epsilon$-factors, $\gamma$-factors and depths.
In other words, for every irreducible smooth representation $\pi$ of $\GL_m(D)$:
\[
\begin{array}{lll}
L (s,\pi) & = & L (s,\rec_{D,m}(\pi)), \\
\epsilon (s,\pi,\psi) & = & \epsilon (s,\rec_{D,m}(\pi),\psi), \\
\gamma (s,\pi,\psi) & = & \gamma (s,\rec_{D,m}(\pi),\psi) \\
d(\pi) & = & d (\rec_{D,m}(\pi)).
\end{array}
\]
\end{thm}
\begin{proof}
It is well-known that the local Langlands correspondence for $\GL_n (F)$
preserves local factors, see the introduction of \cite{HaTa}.

Assume first that $\pi$ is essentially square-integrable. Recall the notations of 
the $\epsilon$ factors of $\pi$ and of $\phi := \rec_{D,m}(\pi) \in\Phi(\GL_m(D))$ 
from \eqref{eq:epsilonPhi} and \eqref{eqn:epsGJ}. By definition
\[
\rec_{D,m}(\pi) = \rec_{F,n}(\JL^{-1} (\pi)) ,
\]
so by Theorem \ref{thm:JL} $\rec_{D,m}$ preserves the $\epsilon$-factors of $\pi$:
\begin{equation} \label{eqn:epsA}
\vareps(0,\phi,\psi) \, q^{-s (a(\phi) + n c(\psi))} = \vareps(s,\phi,\psi) = 
\vareps(s,\pi,\psi) = \vareps (0,\pi,\psi) q^{-s f(\pi,\psi)} .
\end{equation}
Hence, with the notation from Theorem \ref{thm:depth_cond}:
\begin{equation} \label{eqn:a=f}
f(\pi) = f (\pi,\psi) - n c(\psi) = a(\phi) .
\end{equation}
The properties of $\rec_{F,n}$ imply that $\phi$ is elliptic. By combining 
Lemma \ref{lem:phi} with Theorem \ref{thm:depth_cond} and \eqref{eqn:a=f}, 
we obtain that $d(\phi) = d(\pi)$ whenever $\pi$ is essentially square-integrable.

Now let $\pi$ be any irreducible representation of $\GL_m (D)$. By the
Langlands classification, there exist a parabolic subgroup $P \subset \GL_m (D)$ 
with Levi factor $M$ and an irreducible essentially square-integrable 
representation $\omega$ of $M$, such that $\pi$ is a quotient of 
$I_P^{\GL_m (D)} (\omega)$. Moy and Prasad proved in \cite[Theorem 5.2]{MoPr2}
that $\pi$ and $\omega$ have the same depth. By \cite[Theorem 3.4]{Jac} $\pi$
and $\omega$ have the same L-functions and $\epsilon$-factors and by 
\cite[(2.3) and (2.7.3)]{Jac} they also have the same $\gamma$-factors.

On the other hand, $M$ is isomorphic to a product of groups of the form
$\GL_{m_i}(D)$, so the local Langlands correspondence for $M$ is simply the
product of that for the $\GL_{m_i}(D)$. The Langlands parameters $\rec_{D,m}(\pi)$ 
and $\rec_M (\omega)$ are related via an inclusion of the complex dual groups 
$\prod_i \GL_{d m_i}(\C) \to \GL_n (\C)$. Hence these two Langlands parameters 
also have the same depth and local factors.

Because we already proved that the LLC preserves depths for essentially
square-integrable representations of $\GL_m (D)$ or $M$, we can conclude that
\[
d (\pi) = d (\omega) = d (\rec_M (\omega)) = d (\rec_{D,m}(\pi)) ,
\]
and similarly for the local factors.
\end{proof}

\section{The local Langlands correspondence for inner forms of $\SL_n(F)$} 
\label{sec:LLCSL}
\subsection{The statement of the correspondence} \

Recall that $F$ is a non-archimedean local field and that the equivalence
classes of inner forms of $\SL_n (F)$ are in bijection with the isomorphism
classes of central simple $F$-algebras of dimension $n^2$, via
$\Mat_m (D) \mapsto \GL_m (D)_{\der}$. As mentioned after Theorem \ref{thm:1.1},
every character of $Z( \SL_n (\C))$ gives rise to such an algebra via the
Kottwitz isomorphism.

The local Langlands correspondence for
$\GL_m (D)_{\der}$ is implied by that for $\GL_m (D)$, in
the following way. A Langlands parameter 
\[
\phi : \mathbf W_F \times \SL_2 (\C) \to \PGL_n (\C)
\]
which is relevant for $\GL_m (D)_{\der}$ can be lifted it to a Langlands parameter
\[
\overline{\phi} : \mathbf W_F \times \SL_2 (\C) \to \GL_n (\C) 
\]
which is relevant for $\GL_m (D)$, by \cite{Wei}. 
Then $\rec_{m,D}^{-1}(\overline{\phi})$ is an irreducible representation of 
$\GL_m (D)$ which, upon restriction to $\GL_m (D)_{\der}$, decomposes as a finite
direct sum of irreducible representations. The packet $\Pi_\phi (\GL_m (D)_{\der})$
is defined as the set of irreducible constituents of 
$\mathrm{Res}_{\GL_m (D)_\der}^{\GL_m (D)} \rec_{m,D}^{-1}(\overline{\phi})$. 

For these groups it is more interesting to consider the enhanced Langlands 
correspondence, where $\phi$ is supplemented with an irreducible 
representation of a finite group. In addition to the groups defined in
\eqref{eq:1.2}, we write
\begin{equation}\label{eq:1.2bis}
\mc Z_\phi = Z (\SL_n (\C)) / Z (\SL_n (\C)) \cap C(\phi)^\circ \cong
Z (\SL_n (\C)) C(\phi)^\circ / C(\phi)^\circ .
\end{equation}
Any character of $\mc Z_\phi$ determines a character of $Z (\SL_n (\C))$, and 
hence an inner form of $\SL_n (F)$.
An enhanced Langlands parameter is a pair $(\phi,\rho)$ with 
$\rho \in \Irr (\mc S_\phi)$. The groups in \eqref{eq:1.2}, \eqref{eq:1.2bis} 
are related to the more usual component group
\[
S_\phi := Z_{\PGL_n (\C)}(\text{im } \phi) / Z_{\PGL_n (\C)}(\text{im }
\phi)^\circ
\]
by the short exact sequence
\[
1 \to \mc Z_\phi \to \mc S_\phi \to S_\phi \to 1.
\]
Hence $\mc S_\phi$ has more irreducible representations than $S_\phi$. Via
the enhanced Langlands correspondence the additional ones are associated to 
irreducible representations of non-split inner forms of $\SL_n (F)$. The following
result is due to Hiraga and Saito \cite[Theorem 12.7]{HiSa} for generic 
representations of $\GL_m (D)$ when char $F = 0$.

\begin{thm}\label{thm:1.2} \cite[Theorem~1.2]{ABPS3} \\
There exists a bijective correspondence between:
\begin{itemize}
\item pairs $(\GL_m (D)_{\der},\pi)$ with $\pi \in \Irr (\GL_m (D)_{\der})$ and 
$\GL_m (D)_{\der}$ an inner form of $\SL_n (F)$, considered up to equivalence;
\item $\SL_n (\C)$-conjugacy classes of pairs $(\phi,\rho)$ with 
$\phi \in \Phi (\SL_n (F))$ and $\rho \in \Irr (\mc S_\phi)$.
\end{itemize}
Here the group $\GL_m (D)_{\der}$ determines $\rho \big|_{\mc Z_\phi}$ and conversely.
The correspondence satisfies the desired properties from \cite[\S 10.3]{Bor}, with
respect to restriction from inner forms of $\GL_n (F)$, temperedness and essential 
square-integrability of representations.
\end{thm}
We remark that the above bijection need not be canonical if 
$\Pi_\phi (\GL_m (D)_{\der})$ has more than one element.

\subsection{The depth of representations of $\GL_m (D)_\der$} \

For the depth of an irreducible representation of $\GL_m (D)_{\der}$ there are two
candidates. Besides the Moy--Prasad depth one can define the normalized level,
just as in \eqref{eqn:def_depth}. This was done for representations of $\SL_n (F)$
in \cite{BuKu2}. However, Proposition \ref{prop:compareDepth} quickly reveals that 
these two notions agree:

\begin{prop}\label{prop:3.1}
The Moy--Prasad depth of an irreducible representation of \\
$\GL_m (D)_{\der}$ equals its normalized level. 
\end{prop}
\begin{proof}
Let us compare the descriptions of the two kinds of depth with those given in the 
proof of Proposition \ref{prop:compareDepth}. By definition $\GL_m (D)$ and 
$\GL_m (D)_{\der}$ have the same Bruhat--Tits building. The Moy--Prasad depth is 
defined in terms of the groups 
\begin{equation}
P'_{x,r} = P_{x,r} \cap \GL_m (D)_{der} \quad \text{with} \quad x \in \mc B (GL_m (D)).
\end{equation} 
The normalized level is expressed with the groups 
\[
U^j (\mf A)' = U^j (\mf A) \cap \GL_m (D)_{der},
\] 
where $\mf A$ is a hereditary $\mf o_F$-order in $M_m (D)$. With these groups instead 
of $P_{x,r}$ and $U^j (\mf A)$ the entire proof of Proposition 
\ref{prop:compareDepth} carries over to $\GL_m (D)_{der}$.
\end{proof}

It turns out that the depth of an irreducible $\GL_m (D)_\der$-representation $\pi$ 
behaves nicely with respect to restriction from $\GL_m (D)$. To be precise, equals the 
minimum of the 
depths of the irreducible $\GL_m (D)$-representations that contain $\pi$. (Notice that
this minimum is always attained because all depths for inner forms of $\GL_n (F)$ lie
in $\frac{1}{n} \Z$.)

\begin{prop}\label{prop:3.2}
Let $\pi \in \Irr (\GL_m (D)_\der)$ and let $\overline{\pi} \in
\Irr (\GL_m (D))$ be such that 
\begin{itemize}
\item $\pi$ is a direct summand of $\mathrm{Res}_{\GL_m (D)_\der}^{\GL_m (D)} 
(\overline \pi)$;
\item $d(\overline{\pi}) \leq d (\overline{\pi} \otimes \chi \circ \Nrd)$ 
for every character $\chi$ of $F^\times$.
\end{itemize}
Then $d (\pi) =  d(\overline \pi)$. 
\end{prop}
\begin{proof}
In the case $G = \GL_n (F)$, this is guaranteed by Proposition \ref{prop:3.1} and
\cite[Proposition 1.7.iii]{BuKu2}. The same proof works for $\GL_m (D)$ but this would 
be cumbersome, one would have to check that everything in \cite[p.265--268]{BuKu2} 
also works with a division algebra instead of a field. 

Instead, we select some parts of \cite[\S 1]{BuKu2} to provide a shorter proof.
Pick a $x \in \mc B (G)$ such that $(\overline \pi, V)$ has nonzero vector fixed
by $P_{x,d(\overline \pi)+}$. Then
\[
V^{P'_{x,d(\overline \pi)+}} \supset V^{P_{x,d(\overline \pi)+}} \neq 0 ,
\]
so there is an irreducible $\GL_m (D)_{\der}$-subrepresentation $(\pi_1,V_1)$ of 
$\overline \pi$ with 
\[
V_1^{P'_{x,d(\overline \pi)+}} \neq 0 \quad \text{and} \quad 
d(\pi_1) \leq d (\overline \pi).
\]
Since $\overline \pi$ is irreducible, $\pi_1$ is isomorphic to a 
$\GL_m (D)$-conjugate of $\pi$. Conjugation by $g \in \GL_m (D)$ sends any 
Moy--Prasad group $P_{y,r}$ to $P_{g(y),r}$. So this operation preserves depths and
\begin{equation}\label{eq:3.1}
d(\pi) = d(\pi_1) \leq d (\overline \pi) .
\end{equation}
Suppose now that $d(\pi) < d(\overline \pi)$. Take a nonzero $v \in
V^{P'_{x,d(\pi)+}}$ and consider 
\[
V_v := \text{span} \{ \overline \pi (g) v \mid g \in P_{x,d(\pi)+} \} .
\]
As $P'_{x,d(\pi)+}$ is normal in $P_{x,d(\pi)+}$, it acts trivially on $V_v$, and
$V_v$ can be regarded as representation of
\[
P_{x,d(\pi)+} / P'_{x,d(\pi)+} \cong \text{Nrd}(P_{x,d(\pi)+}) \subset F^\times .
\]
Hence there is a character $\chi$ of $F^\times$ such that $\chi^{-1} \circ \text{Nrd}$ 
appears in the action of $P_{x,d(\pi)+}$ on $V_v$. Then the irreducible 
$\GL_m (D)$-representation $\overline \pi \otimes \chi \circ \text{Nrd}$ has a nonzero
vector fixed by $P_{x,d(\pi)+}$, so
\[
d(\overline \pi \otimes \chi \circ \text{Nrd}) \leq d(\pi) < d (\overline \pi) .
\]
This contradicts the assumptions of proposition, so \eqref{eq:3.1} must be an equality.
\end{proof}

\subsection{The depth of Langlands parameters for $\GL_m (D)_\der$} \

The depth of a Langlands parameter $\phi : \bW_F \times \SL_2 (\C) \to
\PGL_n (\C)$ for an inner form of $\SL_n (F)$
is defined as in Section \ref{sec:depth}:
\[
d(\phi) = \inf \{ l \in \R_{\geq 0} \mid \Gal (F_s / F)^{l+} \subset \ker \phi \} . 
\]
The following result may be considered as the non-archimedean analogue of 
\cite[Theorem~1]{ChKa} in the case of the geometric local Langlands correspondence.

\begin{cor}\label{cor:3.3}
Let $\pi \in \Irr (\GL_m (D)_\der)$ with Langlands parameter 
$\phi \in \Phi (\SL_n (F))$. Then $d(\pi) \geq d (\phi)$.
\end{cor}
\begin{proof}
Let $\overline{\pi}$ be as in Proposition \ref{prop:3.2}, so $d(\overline \pi)
= d (\pi)$. Put $\overline \phi = \rec_{D,m}(\overline \pi)$, this is a lift of 
$\phi$ to $\GL_n (\C)$ and Theorem \ref{thm:LLCdepthcusp} says that
$d(\overline \phi) = d (\overline \pi)$. 

We remark that by the compatibility of the LLC with character twists 
\begin{equation}\label{eq:3.2}
d (\overline \phi) \leq d(\overline \phi \otimes \gamma) \text{ for every
character } \gamma \text{ of } \mathbf W_F .
\end{equation}
For any lift $\overline{\phi} \in \Phi (\GL_n (F))$ of $\phi$ we have
$\ker \overline{\phi} \subset \ker \phi$, so $d(\overline{\phi}) \geq d(\phi)$.
\end{proof}

It is possible that the inequality in Corollary \ref{cor:3.3} is strict.
The following example was pointed out to the authors by Mark Reeder.
\begin{ex}
\textup{Take $F = \Q_2$ and a Langlands parameter $\phi : \bW_{\Q_2} \to 
\PGL_2 (\C)$ which is trivial on $\SL_2 (\C)$ and has
image isomorphic to the symmetric group $S_4$. (Such a L-parameter exists,
see for example \cite{Wei}.) We claim that $d(\phi) = 1/3$.\\
Let Ad denote the adjoint representation of $\PGL_2 (\C)$ on 
$\mf{sl}_2 (\C) = \text{Lie}(\PGL_2 (\C))$. Then Ad$\circ 
\phi$ is an irreducible 3-dimensional representation of $\bW_{\Q_2}$. Since
$\PGL_2 (\C)$ is the adjoint group of $\mf{sl}_2 (\C)$, Ad$\circ \phi$ has
the same kernel and hence the same depth as $\phi$. One can check
that Ad$(\phi (\mathbf I_F)) \cong A_4$ and that the image of the wild
inertia subgroup $\mathbf P_F$ is isomorphic to the Klein four group. With 
the formula \cite[(1)]{GrRe} for the Artin conductor we find that
$a(\text{Ad} \circ \phi) = 4$. By Lemma \ref{lem:phi} (with $n=3$) 
$d(\text{Ad} \circ \phi) = 1/3$.\\
Let $\overline \phi : \bW_{\Q_2} \to \GL_2 (\C)$ be a lift of $\phi$. This 
is an irreducible 2-dimensional representation. We claim that 
$d(\overline \phi) \geq 1/2$.\\
With a suitable basis transformation we can achieve that
\[
\phi (\mathbf P_{\Q_2}) = \{ \matje{1}{0}{0}{1}, \matje{0}{i}{i}{0},
\matje{-i}{0}{0}{i}, \matje{0}{1}{-1}{0} \} \subset \PGL_2 (\C) .
\]
Let $1,w_2,w_3,w_4 \in \mathbf P_{\Q_2}$ be preimages of these elements.
Irrespective of the choice of the lift of $\phi$ we have
\[
[\overline{\phi}(w_3),\overline{\phi}(w_4)] = [\matje{-i}{0}{0}{i}, 
\matje{0}{1}{-1}{0}] = \matje{-1}{0}{0}{-1} \in \GL_2 (\C) .
\]
Put $E = F_s^{\ker \overline{\phi}}$ and endow Gal$(E/\Q_2) \cong 
\overline{\phi}(\bW_{\Q_2})$ with the lower numbered filtration. The image of
$\mathbf P_{\Q_2}$ is Gal$(E/\Q_2)_1$ and $[p_3,p_4] \in 
[\text{Gal}(E/\Q_2)_1, \text{Gal}(E/\Q_2)_1]$. By \cite[Propostion IV.2.10]{Ser}
$[p_3,p_4] \in \text{Gal}(E/\Q_2)_3$, so $\overline{\phi}$ is nontrivial on this
ramification group. If we lift $\phi$ with as little 
ramification as possible, $\overline{\phi}(\bW_{\Q_2})$ is an index 2 central
extension of $S_4$. Writing $d_j = |\overline{\phi}(\text{Gal}(E/\Q_2)_j)|$, we have 
\[
d_0 = 24, d_1 = 8, d_2 = d_3 = 2 \text{ and } d_j = 1 \text{ for } j > 3.
\]
The formula \cite[(1)]{GrRe} gives
\[
a(\overline{\phi}) = \frac{\dim (\overline \phi)}{d_0} \sum_{j \geq 0 : d_j > 1} d_j
= \frac{2}{24} (24 + 8 + 2 + 2) = 3 .
\]
Now Lemma \ref{lem:phi} says that $d(\overline \phi) = 1/2$.
}
\end{ex}

To show that Corollary \ref{cor:3.3} is in many cases an equality,
we will make use of several well-known properties of the Artin
reciprocity map $\mathbf a_F : \bW_F \to F^\times$. In particular:

\begin{thm}\label{thm:ramArtin}
$\mathbf a_F (\mathrm{Gal}(F_s / F)^l ) = U_F^{\lceil l \rceil}$ 
for all $l \in \R_{\geq 0}$. 
\end{thm}
\begin{proof}
For any finite abelian extension $E / F$, 
\cite[Corollary 3 to Theorem XV.2.1]{Ser} says that the Artin reciprocity
map gives an isomorphism
\begin{equation}\label{eq:3.7}
\mathbf a_F \colon \text{Gal}(E/F)^l \to U_F^{\lceil l\rceil} / 
\big( \rN_{E/F}(E^\times) \cap U_F^{\lceil l\rceil} \big) .
\end{equation}
Let $F_s
^{\ab}$ be the maximal abelian extension of $F$ contained in $F_s$. 
Taking the projective limit of \eqref{eq:3.7} over all finite, Galois 
subextensions of $F_s^{\ab} / F$, we obtain an isomorphism
\[
\mathbf a_F : \text{Gal}(F_s^{\ab} / F)^l \to  U_F^{\lceil l \rceil} . 
\]
We note that Gal$(F_s^{\ab} / F)$ is the quotient of Gal$(F_s / F)$ modulo
the closure of its commutator subgroup. Hence $\mathbf a_F : \bW_F \to F^\times$
factors via Gal$(F_s^{\ab}/F)$.
\end{proof}

Recall from \cite{BuHe1} that a Langlands parameter for $\GL_n (F)$ is 
\emph{essentially tame} if its restriction to the wild inertia subgroup
$\mathbf P_F$ of $\bW_F$ is a direct sum of characters. 
Clearly $\overline{\phi}$ is essentially tame if and only if $\overline{\phi}
(\mathbf P_F)$ lies in a maximal torus of $\GL_n (\C)$, which in turn is
equivalent to $\phi (\mathbf P_F)$ lying in a maximal torus of $\PGL_n (\C)$.

\begin{defn}
A Langlands parameter $\phi$ for an inner form of $\SL_n (F)$ is essentially
tame if $\phi (\mathbf P_F)$ lies in a maximal torus of $\PGL_n (\C)$. 
\end{defn}

By \cite[Corollary A.4]{BuHe1} any L-parameter for (an inner form of) 
$\GL_n (F)$ is essentially tame if the residual characteristoc of $F$ does not 
divide $n$. Our definition is such that the same holds for Langlands parameters 
for (inner forms of) $\SL_n (F)$.

For such L-parameters the LLC does preserve depths:

\begin{thm}\label{thm:essTameDepth}
Let $\phi \in \Phi (\SL_n (F))$ be essentially tame and relevant for
$\GL_m (D)_\der$. Then $d(\pi) = d(\phi)$ for every 
$\pi \in \Pi_\phi (\GL_m (D)_\der)$.
\end{thm}
\begin{proof} 
Let $\overline{\phi}$ be as in \eqref{eq:3.2}, so $d(\overline \phi) = d(\pi)$.

First we consider the case where $\overline \phi$ is an irreducible 
$n$-dimensional representation of $\bW_F$. By \cite[Theorem A.3]{BuHe1}
there exist a finite, tamely ramified Galois extension $E / F$ and a smooth
character $\xi : \bW_E \to \C^\times$ such that 
$\overline \phi = \mathrm{ind}_{\bW_E}^{\bW_F} \xi$. We may and will assume
that $E$ is contained in our chosen separable closure $F_s$ of $F$.
By Mackey's induction--restriction formula
\[
\mathrm{Res}_{\bW_E}^{\bW_F} (\overline \phi) = \bigoplus\nolimits_{s 
\in \bW_F / \bW_E} \xi^s \text{ , where} \quad \xi^s (w) = \xi (s^{-1} w s) .
\]
The elements of $\bW_F \setminus \bW_E$ permute the $\bW_E$-subrepresentations
$\xi^s$ nontrivially, so they cannot lie in the kernel of $\overline \phi $:
\[
\ker \overline \phi = \{ w \in \bW_E : \xi^s (w) = 1 \; \forall s \in \bW_F \} .
\]
Let pr$: \GL_n (\C) \to \PGL_n (\C)$ be the canonical projection. Then 
$\phi = \mathrm{pr} \circ \overline \phi$ and
\begin{equation}\label{eq:3.3}
\ker \phi = \overline{\phi}^{-1} \big( Z (\GL_n (\C)) \big) = 
\{ w \in \bW_E : \xi^s (w) = \xi (w) \; \forall s \in \bW_F \} .
\end{equation}
Suppose that $d(\overline \phi) > d(\phi)$. In view of the definition of $d(\phi)$,
\begin{equation}\label{eq:3.4}
\ker \overline{\phi} \supset \bW_E \cap \mathrm{Gal}(F_s / F)^{d(\overline \phi)+}
\text{, but } \ker \overline{\phi} \not\supset 
\bW_E \cap \mathrm{Gal}(F_s / F)^{d(\overline \phi)} \subset \ker \phi .
\end{equation}
The relation between the upper and the lower numbering of the filtration subgroups
of $\bW_F$ \cite[\S IV.3]{Ser}, combined with the compatibility of the lower 
numbering with subgroups \cite[Proposition IV.1.2]{Ser}, provides a 
$l \in \R_{\geq 0}$ such that 
\begin{equation}\label{eq:3.5}
\bW_E \cap \mathrm{Gal}(F_s / F)^{d(\overline \phi)} = \mathrm{Gal}(F_s / E)^l .
\end{equation}
In fact $l > 0$ because $d(\overline \phi) > d (\phi) \geq 0$.
Since $\mathrm{Res}^{\bW_F}_{\bW_E} \overline{\phi}$ is a direct sum of characters, 
it factors through the Artin reciprocity map $\mathbf a_E : \bW_E \to E^\times$. With
\eqref{eq:3.4} we see that
\[
\mathbf a_E ( \mathrm{Gal}(F_s / E)^l ) \neq \mathbf a_E (\mathrm{Gal}(F_s / E)^{l+} ) .
\]
By Theorem \ref{thm:ramArtin} applied to $F_s / E ,\; l$ must be a positive integer. 
When we transfer the conjugation action
of $\bW_F$ on $\bW_E$ to $E^\times$ via Artin reciprocity, it becomes the standard
action of Gal$(E/F) \cong \bW_F / \bW_E$ on $E^\times$. Now \eqref{eq:3.3} says
that $\xi$ is a Gal$(E/F)$-invariant character of $U_E^l$. Since $l \in \Z_{>0}$
and $E / F$ is tamely ramified, $U_E^l$ is a cohomologically trivial Gal$(E/F)$-module.
According to \cite[Lemma A.1]{BuHe1}, these properties imply that $\xi$ factors
through the norm map $\rN_{E/F}$, and there is a unique smooth character 
\[
\xi' \text{ of } U_E^l \cap F^\times = \rN_{E/F}(U_E^l) \text{ such that }
\xi = \xi' \circ \rN_{E/F} \text{ on } U_E^l.
\]
Since $F^\times / \ker (\xi')$ is a finitely generated abelian group and $\C^\times$
is divisible, we can extend $\xi'$ to a smooth character of $F^\times$. Via Artin
reciprocity this yields a character $\xi_F$ of $\bW_F$. From \eqref{eq:3.5} and 
the commutative diagram \cite[\S XI.3]{Ser}
\[
\xymatrix{
\bW_E \ar[r] \ar[d]_{\mathbf a_E} & \bW_F \ar[d]^{\mathbf a_F} \\
E^\times \ar[r]^{\rN_{E/F}} & F^\times
}
\]
we see that $\xi_F = \xi$ on Gal$(F_s / E)^{d(\overline \phi)} \cap \bW_E$. Then
$\overline \phi \otimes \xi_F^{-1}$ is another lift of $\phi$ to $\Phi (\GL_n (F))$,
and $\ker \overline \phi \otimes \xi_F^{-1}$ contains 
Gal$(F_s / E)^{d (\overline \phi)}$. Thus $d(\overline \phi \otimes \xi_F^{-1}) <
d(\overline \phi)$, which contradicts the definition of $\overline \phi$. We have
shown that $d(\overline \phi) = d(\phi)$ if $\overline{\phi}|_{\bW_F}$ is irreducible.

For a general essentially tame parameter $\overline \phi$ for $\GL_n (F) ,\;
\overline{\phi} |_{\bW_F}$ is a direct sum of irreducible essentially tame 
parameters $\psi_i$ for $\GL_{n_i}(F)$, with $n_i \leq n$. Writing $\psi_i = 
\mathrm{ind}_{\bW_{E_i}}^{\bW_F} \xi_i$ as above, we obtain from \eqref{eq:3.3} that
\begin{equation}\label{eq:3.8}
\ker \phi = \{ w \in \cap_i \bW_{E_i} : \xi_i^s (w) = \xi_j (w)
\text{ for all } i,j \text{ and all } s \in \bW_F \} .
\end{equation}
In general this is smaller than
\[
\ker (\oplus_i \mathrm{pr}_i \circ \psi_i) =  \{ w \in \cap_i \bW_{E_i} : 
\xi_i^s (w) = \xi_i (w) \text{ for all } i \text{ and all } s \in \bW_F \} ,
\]
where $\mathrm{pr}_i : \GL_{n_i}(\C) \to \PGL_{n_i} (\C)$ denotes the canonical 
projection. Comparing all these kernels we deduce that
\begin{equation}\label{eq:3.6}
\max_i d(\psi_i) = d (\overline \phi) \geq d (\phi) \geq
d(\oplus_i \mathrm{pr}_i \circ \psi_i) = \max_i d (\mathrm{pr}_i \circ \psi_i ) .
\end{equation}
However, we cannot just twist $\overline \phi$ with a character of $\bW_F$ derived 
from the most ramified of the $\psi_i$ as in the irreducible case, because that 
could make the depth of another $\psi_j$ much larger. 

We suppose once again that $d(\overline \phi) > d(\phi)$. Then 
\[
\ker \overline{\phi} \supset \mathrm{Gal}(F_s / F)^{d(\overline \phi)+}
\text{, but } \ker \overline{\phi} \not\supset 
\mathrm{Gal}(F_s / F)^{d(\overline \phi)} \subset \ker \phi . 
\]
By \eqref{eq:3.8} all the $\xi_i$
agree on Gal$(F_s / F)^{d(\overline \phi)}$. The above method produces characters
$\xi'_i$ of $U_{E_i}^{d_i} \cap F^\times$, which agree on $\mathbf a_F \big( 
\bW_F^{d(\overline \phi)} \big)$. Put $\xi' = \xi'_i |_{\mathbf a_F \big( 
\bW_F^{d(\overline \phi)} \big)}$. Now the same argument as in the irreducible
case leads to a contradiction with \eqref{eq:3.2}. 
\end{proof}

\end{document}